\documentclass[11pt,a4paper]{amsart}

\usepackage[colorinlistoftodos, textsize=tiny]{todonotes}
\makeatletter
\providecommand\@dotsep{5}
\def\listtodoname{}
\def\listoftodos{\@starttoc{tdo}\listtodoname}
\makeatother

\usepackage[utf8]{inputenc}
\usepackage[english]{babel}
\usepackage{amsmath}
\usepackage{fullpage}
\usepackage{todonotes}
\usepackage{enumerate}
\usepackage[all,cmtip]{xy}
\usepackage{hyperref}

\newtheorem{theorem}{Theorem}

\newtheorem{corollary}[theorem]{Corollary}

\newtheorem{definition}[theorem]{Definition}

\newtheorem{lemma}[theorem]{Lemma}

\newtheorem{proposition}[theorem]{Proposition}

\numberwithin{theorem}{section}

\theoremstyle{definition}
\newtheorem{remark}[theorem]{Remark}
\newtheorem{algorithm}[theorem]{Algorithm}

\usepackage[normalem]{ulem}

\newcommand{\abGal}[1]{\operatorname{Gal}\left(\overline{#1}/#1\right)}

\newcommand{\Aut}{\operatorname{Aut}}
\newcommand{\Gal}{\operatorname{Gal}}
\newcommand{\GL}{\operatorname{GL}}

\usepackage[all,cmtip]{xy}

\begin{document}

\title{Computing twists of hyperelliptic curves}
\author{Davide Lombardo, Elisa Lorenzo Garc\'ia}

\address{Davide Lombardo, Università di Pisa, Largo Bruno Pontecorvo 5, 56127 Pisa, Italy}
\email{davide.lombardo@unipi.it}

\address{Elisa Lorenzo Garc\'ia, IRMAR, Universit\'e de Rennes 1, Campus de Beaulieu, 35042  RENNES Cédex, France}
\email{elisa.lorenzogarcia@univ-rennes1.fr}


\date{}

\begin{abstract}
We give an efficient algorithm to compute equations of twists of hyperelliptic curves $C$ of arbitrary genus over any perfect field $k$ (of characteristic different from 2) starting with a cocycle in $\operatorname{H}^1(\operatorname{Gal}(\overline{k}/k),\operatorname{Aut}(C))$. We also discuss some interesting examples. 
\end{abstract}
\maketitle

\section{Introduction}

In this paper we describe an algorithm to compute explicit equations of twists of hyperelliptic curves (of genus $g \geq 2$) starting from a cocycle in a given cohomology class.
The study of twists of curves is a very useful tool for understanding some arithmetic problems:
for example, it has proved to be extremely helpful to explore the Sato-Tate conjecture \cite{AWS2016,FLS,FS}, as well as to solve some Diophantine equations \cite{PSS}, to compute $\mathbb{Q}$-curves realizing certain Galois representations \cite{BFGL}, to find counterexamples to the Hasse principle \cite{Oz}, \cite{Oz2}, and to produce curves with many points \cite{MT}. 

Let us briefly review previous work concerning the computation of twists of curves.
Geometrically, all smooth conics are isomorphic to $\mathbb{P}^1$, so the twists of a smooth conic $C/K$ are precisely the smooth conics over $K$ (recall that the anticanonical embedding realizes any genus-0 curve as a conic). The set of twists of an elliptic curve is well understood if one restricts to those automorphisms that fix the origin of the group law (see for example \cite[Section X.5]{Sil} and \cite{Top}), but the situation is way more complicated if one considers arbitrary twists. 
For curves of genus $2$, we refer to the work of Cardona \cite{Cart} and for non-hyperelliptic curves of genus $3$ to the work of the second author \cite{Lor16}. A general algorithm to compute twists of non-hyperelliptic curves is also due to the second author \cite{Lor15,Tesis}, and therefore only the general hyperelliptic case remains to be treated: thanks to the algorithm presented in this paper, we are now able to compute equations of twists of any curve of genus $g \geq 2$ (in characteristic different from 2) starting from the corresponding cocycles.

{
We remark that the main difficulty of working with hyperelliptic curves is that $\operatorname{Aut}_{\overline{k}}(C)$ does not embed in the automorphism group of $H^0(C,\Omega^1_C)$; indeed, this is the property that makes the construction of \cite{Lor15,Tesis} work, and it fails for hyperelliptic curves because the canonical bundle is not very ample in this case.
To remedy this, one could consider higher powers of the canonical bundle: given any hyperelliptic curve $C/k$ of genus $g$, the line bundle $(\Omega^1_C)^{\otimes 2}$ is very ample, so that we get an embedding of $C$ into $\mathbb{P}H^0\left(C, (\Omega^1_C)^{\otimes 2}\right)^\vee \cong \mathbb{P}^{3(g-1)-1}$. Moreover, $\operatorname{Aut}(C)$ embeds into $\operatorname{GL}\left( H^0\left(C, (\Omega^1_C)^{\otimes 2}\right)\right)$, and this would allow us to (in principle) use the approach of \cite{Lor15,Tesis} to compute models of twists of $C$. However, this seems quite inefficient from a computational point of view, since it requires one to explicitly find a basis of sections for $(\Omega^1_C)^{\otimes 2}$, and to work with models of $C$ in high-dimensional projective spaces. Our algorithm, which exploits the specific geometry of the situation, provides a much more efficient approach.}

We briefly describe the organization of the paper: in Section \ref{sect:HypCurves} we set up our notation for hyperelliptic curves and recall some well-known facts about them. Section \ref{method} forms the core of the paper, giving the details of our algorithm for the computation of twists.
Finally, in Section \ref{sect:Examples} we compute equations for several twists given by cocycles.

\subsection*{Notation}
Throughout the paper $k$ is a perfect field of characteristic different from $2$, with algebraic closure $\overline{k}$, and $\mu_n(\overline{k})$ is the set of $n$-th roots of unity in $\overline{k}$. We denote by $\zeta_n$ a fixed primitive $n$-th root of unity in $\overline{k}$ (provided that $(n,\operatorname{char}(k))=1$) and by $\operatorname{Gal}(\overline{k}/k)$ the absolute Galois group of $k$.
If $M$ is a matrix in $\GL_3(k)$, we write $M(x,y,z)$ for the linear map that sends $x, y, z$ respectively to the three entries of the vector $M \cdot (x, y, z)^t$; here $x,y,z$ can be either variables or elements of $k$.

By a curve $C/k$ we mean an irreducible, geometrically connected, projective variety of dimension 1 defined over a field $k$. We denote by $\Aut(C)$ the group of geometric automorphisms of $C$, that is, the group of automorphisms of $C \times_k {\overline{k}}$, and if $L/k$ is a field extension we denote by $\Aut_L(C)$ the group of automorphisms of $C$ defined over $L$.

\subsection*{Acknowledgment}
We thank Christophe Ritzenthaler for useful
discussions and for pointing out some of the references, { and Jeroen Sijsling for his careful reading of a preliminary version of this manuscript and for his many useful comments. We also thank Bas Edixhoven and Bjorn Poonen for interesting discussions during the workshop Arithmetic Geometry and Computer Algebra held in Oldenburg in June 2017. Finally, we thank the anonymous referee for helping us improving the clarity of the exposition and some of the proofs in the paper, especially Lemma \ref{superembedding0} and Corollary \ref{superembedding}.

Part of this work was carried out
during a visit of the first author to Universit\'e de Rennes 1, and we are grateful to this institution for its hospitality and for the ideal working conditions.

\subsection{Twists and cohomology}
We recall the notion of twist of a curve and its well-known relationship with a certain Galois cohomology set.

\begin{definition}
Let $C/k$ be a smooth projective curve. A twist of $C/k$ is 
a smooth projective curve $C'/k$ 
for which there exists a
$\overline{k}$-isomorphism $\varphi:\,C' \to C$. We identify two twists if they are isomorphic over $k$.
\end{definition}

\begin{theorem}[\cite{Sil}, Chapter X, Theorem $2.2$]\label{thm_CohomologyTwists}
The set of twists of $C/k$ is in bijection with the pointed cohomology set $\operatorname{H}^1\left(\abGal{k},\Aut(C) \right)$. The bijection is given explicitly as follows:
\begin{itemize}
\item let $C'/k$ be a twist of $C/k$ with associated $\overline{k}$-isomorphism $\varphi: C' \to C$. The corresponding cohomology class is that of the cocycle
\[
\begin{array}{cccc}
\xi : & \abGal{k} & \to & \Aut(C) \\
 & \sigma & \mapsto & \varphi \circ\,^\sigma \varphi^{-1}.
\end{array}
\]
\item let $\xi \in \operatorname{H}^1\left(\abGal{k}, \Aut(C) \right)$. We define an action of $\abGal{k}$ on $\overline{k}(C)$ extending by linearity the prescription
\[
\sigma:\,a f \mapsto \sigma(a) \cdot \xi(\sigma)(f) \quad \forall a \in \overline{k}, \forall f \in k(C).
\]
Given $\xi$,  we obtain a twist $C'$ as follows: the curve $C'$ is the unique smooth projective $k$-curve whose function field is $\overline{k}(C)^{\abGal{k}}$, the field of invariants for the action just defined. The $\overline{k}$-isomorphism $\varphi:\,C' \to C$ comes from the isomorphism $\overline{k}(C') \cong \overline{k}(C)$ induced by the natural inclusion $k(C') \hookrightarrow \overline{k}(C)$.
\end{itemize}
We say that a $\overline{k}$-isomorphism $\varphi : C' \to C$ realizes the cocycle $\xi : \operatorname{Gal}(\overline{k}/k) \to \operatorname{Aut}(C)$ if $\varphi \circ\,^\sigma\varphi^{-1} = \xi(\sigma)$ for all $\sigma \in \operatorname{Gal}(\overline{k}/k)$.
\end{theorem}

\section{Hyperelliptic curves}\label{sect:HypCurves}
By a hyperelliptic curve over a field $k$ we mean a smooth projective curve $C$ over $k$, of positive genus $g$, and such that the canonical morphism maps $C$ to a genus-0 curve. {We observe that curves of genus 1 are not hyperelliptic with this definition, because for such curves the canonical morphism maps $C$ to a point.}
Also notice that some authors might call the curves we consider \textit{geometrically hyperelliptic}, reserving the term hyperelliptic for curves given by a hyperelliptic model over $k$ (see Equation \eqref{eq:HyperellipticModel} below). We also remark that we shall often use affine models of (Zariski-open subsets of) $C$: given an affine curve, there is, up to isomorphism, precisely one smooth projective curve with the same function field, so we can make this identification without loss of generality (as our base field is perfect, we need not to worry about the distinction between smooth and regular).

If $C$ is a hyperelliptic curve, one knows that the canonical morphism is 2-to-1, and $C$ admits an involution $\iota :C \to C$ (the hyperelliptic involution) that preserves its fibers. When the genus-0 quotient $C/\langle \iota \rangle$ is isomorphic to $\mathbb{P}^1$ over $k$ (and not just over $\overline{k}$), and provided that $\operatorname{char}(k) \neq 2$, the curve $C$ admits a $k$-model of the form
\begin{equation}\label{eq:HyperellipticModel}
y^2=f(x),
\end{equation}
where $f(x)$ is a polynomial of degree $2g+1$ or $2g+2$. A model as in \eqref{eq:HyperellipticModel} will be called a hyperelliptic model over $k$. Notice that not all hyperelliptic curves defined over $k$ admit a hyperelliptic model over $k$; however, a result of Mestre shows that this is the case for all hyperelliptic curves of even genus.

\begin{lemma}[{\cite[§2.1]{MR1106431}}]\label{isos}
 Let $C/k$ be a hyperelliptic curve (where $\operatorname{char}(k) \neq 2$). If the genus of $C$ is even, then $C$ admits a hyperelliptic model defined over $k$. \end{lemma} 

\begin{remark}\label{rem-fields}
	Over finite fields, the field of moduli of any hyperelliptic curve is a field of definition and moreover of hyperelliptic definition \cite{LR, LRS}. In characteristic zero, for elliptic curves we know that the field of moduli is always a field of definition and of hyperelliptic definition. For genus $2$, and for any hyperelliptic curve of even genus, the fields of definition and hyperelliptic definition coincide, but they are not necessarily equal to the field of moduli (this is measured by Mestre's obstruction) \cite{LR, LRS}. We find the first example for which a field of definition is not a field of hyperelliptic definition in the genus $3$ case \cite{Hug}.
\end{remark}

Even when a hyperelliptic model is not available (necessarily in the odd genus case), the fact that the quotient $C/\langle \iota \rangle$ is of genus 0 implies the existence of a specific kind of $k$-model for $C$, which we now describe:

\begin{lemma}\label{hypmodel}
	Let $C/k$ be a hyperelliptic curve of odd genus $g$ (over a field $k$ of characteristic different from 2). There exists a $k$-rational model
	\begin{equation}\label{eq:NonHyperellipticModel}
		C:\,
	\begin{cases}
	t^2=f(x,y,z)\\z^2=ax^2+by^2
	\end{cases} \subseteq \mathbb{P}_{1,1,1,\frac{g+1}{2}}(k),
	\end{equation}
	where $f\in k[x,y,z]$ is a homogeneous polynomial of degree $g+1$ and $\mathbb{P}_{1,1,1,\frac{g+1}{2}}(k)$ is a weighted projective space over $k$ (in the variables $x,y,z,t$, with weights $1,1,1,\frac{g+1}{2}$). Moreover, $C$ has a hyperelliptic model over $k$ if and only if the quaternion algebra $\left( \frac{a,b}{k}\right)$ is trivial. 
\end{lemma}

\begin{proof}
Let $\iota$ be the hyperelliptic involution of $C$. 
By definition of a hyperelliptic curve, the canonical morphism gives a 2-to-1 map $C \to C/\langle \iota \rangle$. As $C/ \langle \iota \rangle$ has genus $0$, it is $k$-isomorphic to a curve given by a homogeneous equation of the form $ax^2+by^2=z^2$ for some $a,b\in k$. The quotient map has degree $2$, so $C$ is given by a model
	\[C:\,
	\begin{cases}
	t^2=f(x,y,z)\\z^2=ax^2+by^2
	\end{cases}.
	\]
	Since $C$ has genus $g$, the map $C \to C /\langle \iota \rangle $ ramifies at $2g+2=2\cdot\text{deg}(f)$ points, and $f$ must have degree $g+1$. This establishes the first part of the lemma.
	
As for the second part, recall that a curve of genus $\geq 2$ {admits at most one 2-to-1 map to a genus 0 curve (that is, $C$ is hyperelliptic in at most one way)}. If $C/k$ admits a hyperelliptic model over $k$, then by definition it {admits a 2-to-1 map to $\mathbb{P}^1_k$}, and since it also {admits a 2-to-1 map to $z^2=ax^2+by^2$} this latter quadric must be isomorphic to $\mathbb{P}^1$ over $k$, which happens precisely when the quaternion algebra $\left( \frac{a, b}{k}\right)$ is trivial.

Conversely, suppose that the quaternion algebra is trivial: then $z^2=ax^2+by^2$ has a rational point, so we can find three new variables $x', y', z'$ which are (invertible) linear functions of $x,y,z$ and satisfy $z'^2 = x'y'$. In terms of these new variables, $C$ is given by a model
\[
	\begin{cases}
	t^2=F \left(x',y',z'\right)\\ z'^2=x'y'
	\end{cases}.
	\]
Restricting to the affine chart $z'=1$ and replacing $y'$ by $1/x'$ in the first equation we find the model
\[
t^2 = F\left( x', 1/x', 1 \right);
\]
multiplying by $x'^{g+1}$ and setting $w:=t x'^{(g+1)/2}$ we finally get the hyperelliptic model
\[
w^2 = t^2 x'^{g+1} = x'^{g+1} F\left( x', 1/x', 1 \right) =: H(x').
\]
\end{proof}

In the course of the proof of Lemma \ref{hypmodel} we have seen how to construct a hyperelliptic model for a (odd genus) curve given by a model as in \eqref{eq:NonHyperellipticModel}, if such a hyperelliptic model exists. For the inverse transformation, if we start with $C \, : y^2=f(x)=\sum_{n=0}^{2g+2} a_nx^n$, with $a_{2g+2}$ not necessarily nonzero, then $C$ admits the following model of the form \eqref{eq:NonHyperellipticModel}:
\begin{equation}\label{bigmodel}
\begin{cases}
t^2=\sum_{n=0}^{g+1} a_n y'^{g+1-n}z'^n + \sum_{n=g+2}^{2g+2} a_n x'^{n-g-1}z'^{2g+2-n}
\\
z'^2 = x'y'
\end{cases}.
\end{equation}

\subsection{Automorphisms of hyperelliptic curves}\label{sect:Automorphisms}

We start by considering automorphisms of curves given by hyperelliptic models.
It is well-known that if the genus $g$ is at least 2, then 
all the isomorphisms from a curve $C$ to a curve $C'$, both given by hyperelliptic models as in \eqref{eq:HyperellipticModel}, are of the form
\begin{equation}\label{iso0}
\phi : (x,y)\mapsto \left(\frac{\alpha x+\beta}{\gamma x+\delta},y\frac{e}{(\gamma x+\delta)^{g+1}}\right)
\end{equation}
for some $\alpha,\,\beta,\,\gamma,\,\delta,\,e\in\bar{k}$ (see for example \cite[Proposition 3.1.1]{Hug} or \cite[§1.5.1]{MR3207427}).

\begin{lemma}\label{superembedding0}
Let $r$ be an integer different from $\frac{g+1}{2}$ and set $D=|g+1-2r|$.
\begin{enumerate}[(a)]
\item For any automorphism $\phi$ of $C$ there exist $\alpha', \beta', \gamma', \delta' \in \overline{k}$ such that $\phi$ can be represented as
	\begin{equation}\label{iso}
	\phi:(x,y)\mapsto \left(\frac{\alpha' x+\beta'}{\gamma' x+\delta'},y\frac{(\alpha'\delta'-\beta'\gamma')^r}{(\gamma' x+\delta')^{g+1}}\right).
	\end{equation}
The matrix $\begin{pmatrix}
\alpha' & \beta' \\ \gamma' & \delta'
\end{pmatrix}$ is well-defined up to multiplication by a $D$-th root of unity in $\overline{k}$.
\item The map 
\[
\begin{array}{cccc}
\operatorname{Aut}(C) & \to & \operatorname{GL}_2(\overline{k})/\mu_D(\overline{k}) \\
\phi & \mapsto & \begin{pmatrix}
\alpha' & \beta' \\
\gamma' & \delta'
\end{pmatrix},
\end{array}
\]
where $\alpha',\beta',\gamma',\delta'$ are as above, is a Galois-equivariant embedding.
\end{enumerate}
\end{lemma}

\begin{proof} For part (a), start with a representation of $\phi$ as in Equation \eqref{iso0}.
 Set $\Delta=\alpha\delta-\beta\gamma$ and take $(\alpha',\,\beta',\,\gamma',\,\delta')=\lambda(\alpha,\,\beta,\,\gamma,\,\delta)$, where $\lambda$ is any solution to the equation $\lambda^{2r-g-1}=e\times\Delta^{-r}$. Uniqueness of the representation up to $D$-th roots of unity and part (b) of the lemma are immediate to check.
\end{proof}

\begin{corollary}\label{superembedding}
	Let $C$ be a hyperelliptic curve of genus $g$ admitting a hyperelliptic model over $k$. If $g$ is odd, there exists a Galois-equivariant embedding 
	$$
	\operatorname{Aut}(C)\hookrightarrow\operatorname{GL}_2(\overline{k})/\{\pm1\};
	$$
	if $g$ is even, there exists a Galois-equivariant embedding 
	$$
	\operatorname{Aut}(C)\hookrightarrow\operatorname{GL}_2(\overline{k}).
	$$
\end{corollary}

\begin{proof} It suffices to take $r=(g-1)/2$ (for $g$ odd) and $r=g/2$ (for $g$ even) in the previous lemma.
\end{proof}

\medskip

We now discuss the case of $C$ not being defined by a hyperelliptic model. Let
	$$C:\,
	\begin{cases}
	t^2=f(x,y,z)\\z^2=ax^2+by^2
	\end{cases}
	$$
be a hyperelliptic curve of odd genus $g$ over a field $k$. To describe the automorphisms of $C$, we start by recalling the fundamental fact that the hyperelliptic involution is unique and lies in the center of $\operatorname{Aut}(C)$, see \cite[III.7.9]{RiemannSurfaces}. Notice that any automorphism of $C$ induces an automorphism of $C / \langle \iota \rangle =: \mathcal{L}$. It is easy to see that $\mathcal{L}$ can be embedded in $\mathbb{P}^2$, and every automorphism of $\mathcal{L}$ lifts to an automorphism of $\mathbb{P}^2$: one way to check this is to identify $\mathcal{L}$ with its anticanonical embedding, and notice that every automorphism of $\mathcal{L}$ sends the anticanonical sheaf $O(2)$ to itself, hence it induces a well-defined automorphism of $\mathbb{P}H^0(C,O(2))^\vee=\mathbb{P}^2$. 
Hence for every $k$-automorphism $\varphi$ of $C$ we have a corresponding element of $\operatorname{PGL}_3(k)=\Aut(\mathbb{P}^2_k)$, and in particular on the (homogeneous) coordinates $[x,y,z]$ the action of $\varphi$ is given by
\[
[x,y,z] \mapsto \left[ M_{11}x+M_{12}y+M_{13}z, M_{21}x+M_{22}y+M_{23}z, M_{31}x+M_{32}y+M_{33}z \right]
\]
for some $M=(M_{ij})_{1 \leq i,j \leq 3} \in \operatorname{GL}_3(k)$. {Using the fact that the hyperelliptic involution is unique, and hence stable under $\varphi$,} we also have $\iota^* \varphi^*(t)=-\varphi^*(t)$, so (since $\iota$ acts trivially on $x,y,z$ and sends $t$ to $-t$) we must have $\varphi^*(t) = t g(x,y,z)$ for some $g(x,y,z)$. Since the image of $[x,y,z,t]$ must still lie on $C$ we easily get the condition $\varphi^*(t)=et$ for some $e \in k^\times$: indeed, we must have $\varphi^*(t)^2 = f(\varphi^* (x),\varphi^* (y),\varphi^* (z))$; comparing the degree in $x,y,z$ of the two sides of this equality shows that $g$ is a constant.

\section{Computing equations of twists}\label{method}
In this section we describe our algorithms to compute twists of hyperelliptic curves and prove their correctness. In what follows $k$ is a perfect field of characteristic different from 2. We first treat the even genus case, which is easier than the odd genus one since (thanks to Lemma \ref{isos}) we always have a hyperelliptic model over the field of definition for the curve.

\subsection{The even genus case}\label{sect:EvenGenus}
The algorithm is the following.

\medskip

\noindent\fbox{\begin{minipage}{39.5em}
		\vspace{2mm}
		\begin{algorithm}\label{algo:MainEvenGenus} (For computing twists of hyperelliptic curves of even genus)
			
			\vspace{5mm}
			
			\textbf{Input:} A hyperelliptic curve $C/k$ of even genus given by a hyperelliptic equation $y^2=f(x)$, and a cocycle $\xi:\operatorname{Gal}(\overline{k}/k) \to \Aut(C)$ factoring through a finite extension $L/k$.
			
			\textbf{Output:} A curve $C'/k$ and a $\overline{k}$-isomorphism $\phi:\,C'\rightarrow C$ such that  $\xi(\sigma)=\phi\circ\,^{\sigma}\phi^{-1}$ for all $\sigma \in \Gal(\overline{k}/k)$.
			
			\vspace{5mm}
			
			\begin{enumerate}
				\item Compose $\xi$ with the embedding $\Aut(C) \hookrightarrow \GL_2(\overline{k})$ given by Corollary \ref{superembedding} to obtain a new cocycle $\psi$ in $\operatorname{Z}^1(\operatorname{Gal}(\overline{k}/k),\GL_2(\overline{k} ))$.
				\item Using Hilbert's theorem 90, compute $M \in \operatorname{GL}_2(\overline{k})$ such that ${\psi}(\sigma)=M \cdot \,^\sigma M^{-1}$.
				\item Write $M=\begin{pmatrix}
\alpha & \beta \\ \gamma & \delta
\end{pmatrix}$ and let
$
C' : y^2 = \det(M)^{-g} (\gamma x+\delta)^{2g+2} f\left( \frac{\alpha x+\beta}{\gamma x+\delta} \right).
$
Define
\[
\begin{array}{cccc}
\phi : & C' & \to & C \\
& (x,y) & \mapsto & \left(\frac{\alpha x+\beta}{\gamma x+\delta}, y\frac{(\alpha\delta-\beta\gamma)^{g/2}}{(\gamma x+\delta)^{g+1}} \right)
\end{array}
\]
$C'/k$ is the desired twist and the isomorphism $\phi:\,C'\rightarrow C$ realizes the cocycle $\xi$. 
				\vspace{5mm}
			\end{enumerate}
		\end{algorithm}
	\end{minipage}}
	
	\vspace{5mm}
	
In the light of Corollary \ref{superembedding} the correctness of the algorithm is trivial (one checks easily that the model of $C'$ thus found has $k$-rational coefficients). From the computational point of view, the important issues are computing the lift of the cocycle in step (1), for which see the proof of Lemma \ref{superembedding0}, and the computation of $M$ using Hilbert's Theorem $90$, see subsection \ref{sect:Hilbert90}.

\subsection{The odd genus case}\label{sect:OddGenus} Let $C$ be a hyperelliptic curve of odd genus $g$ defined over a perfect field of characteristic different from $2$. We now describe our approach to computing twists of $C$ in this case, starting -- in the next subsection -- with the problem of finding models of twists of conics.




\subsubsection{Computing twists of conics}\label{sect:TwistingConics}
Our approach relies on two fundamental ideas. 
The first is that working with the anti-canonical embedding of a conic $Q/k$ allows us to interpret cocycles with values in $\Aut(Q)$ as cocycles with values in $\operatorname{PGL}_3(\overline{k})$. The second is that -- as we shall show -- such cocycles can then be lifted to cocycles with values in $\GL_3(\overline{k})$, thus removing the ambiguity coming from the projective quotient. Notice that it is not true in general that one can lift cocycles with values in $\operatorname{PGL}_3(\overline{k})$ to cocycles with values in $\GL_3(\overline{k})$, but this will turn out to be the case for the specific cocycles we need to work with.


\begin{theorem}\label{thm-cle}
Let $Q$ be a smooth genus-0 curve over a field $k$, embedded as a conic section in $\mathbb{P}^2_k$. There is an injective, Galois-equivariant map $\Aut(Q) \hookrightarrow \Aut(\mathbb{P}_{\overline{k}}^2) \cong \operatorname{PGL}_3(\overline{k})$. Furthermore, any cocycle $\xi : \Gal(\overline{k}/k) \to \Aut(Q) \hookrightarrow \operatorname{PGL}_3(\overline{k})$ 
can be lifted to a cocycle $\tilde{\xi}: \Gal(\overline{k}/k) \to \operatorname{GL}_3(\overline{k})$.
\end{theorem}
\begin{proof}
The conic $Q$ is embedded in $\mathbb{P}^2$ via the linear series corresponding to $O(2)$. Any automorphism of $Q_{\overline{k}}$ pulls $O(2)$ back to itself (since there is only one equivalence class of divisors of degree 2), hence it induces an automorphism of $\mathbb{P}^2_{\overline{k}}$ which, upon restriction to $Q$, is the automorphism we started with. This proves the first statement.

For the second part, notice that given a class in $\operatorname{PGL}_3(\overline{k})$ there are only two matrices $\pm M \in \GL_3(\overline{k})$ representing that class and whose action on the ring $\overline{k}[x,y,z]$  preserves the equation of $Q$. This gives us a way to uniquely determine an element of $\operatorname{GL}_3(\overline{k})/\mu_2(\overline{k})$ from an element of $\Aut(Q) \subseteq \operatorname{PGL}_3(\overline{k})$, thus providing a lift of the embedding $\Aut(Q) \hookrightarrow \operatorname{PGL}_3(\overline{k})$ to $\Aut(Q) \hookrightarrow \operatorname{GL}_3(\overline{k})/\mu_2(\overline{k})$.
Finally we notice that $\operatorname{GL}_3(\overline{k})/\mu_2(\overline{k})\simeq\operatorname{GL}_3(\overline{k})$ by the isomorphism  $A\mapsto\frac{1}{\operatorname{det}(A)}A$. Now it is easy to check that the composition $\Aut(Q) \hookrightarrow \operatorname{GL}_3(\overline{k})$ is Galois-equivariant.
\end{proof}

\begin{remark}\label{CoolEmbedding} If $Q=\mathbb{P}^1$ and it is embedded in $\mathbb{P}^2$ as $Q_0:\,z^2=xy$, the previous lift is given by the map $\psi:\,\operatorname{Aut}(\mathbb{P}^1)\simeq \operatorname{PGL}_2(\overline{k}) \hookrightarrow\operatorname{GL}_3(\overline{k})$ defined by
$$
\left[\begin{pmatrix} \alpha&\beta \\ \gamma& \delta \end{pmatrix}\right]\mapsto\frac{1}{\alpha\delta-\beta\gamma}\begin{pmatrix} \alpha^2& \beta^2 & 2\alpha\beta \\ \gamma^2 & \delta^2 & 2\gamma\delta \\ \alpha\gamma & \beta\delta & \beta\gamma+\alpha\delta\end{pmatrix}.
$$
\end{remark}

\begin{remark}\label{explicit-lift}
	In general, if $Q:\,ax^2+by^2=z^2$, we consider the isomorphism $$B=\begin{pmatrix}\sqrt{a} & \sqrt{-b} & 0\\
	\sqrt{a} & -\sqrt{-b} & 0\\
	0 & 0 & 1
	\end{pmatrix}:\,Q\rightarrow Q_0.$$ The embedding of $Q$ in $\mathbb{P}^2$ in Theorem \ref{thm-cle} can then be obtained by choosing as basis of $\operatorname{H}^0(Q_{\overline{k}},O(2))$ the three sections $\begin{pmatrix}
	x \\ y \\ z
	\end{pmatrix} = B^{-1} \begin{pmatrix}
	(dt)^* \\ t^2 (dt)^* \\ t(dt)^*
	\end{pmatrix}$; notice that by definition of $B$ we have $ax^2+by^2=z^2$, so $Q$ is again identified with its anticanonical embedding. The corresponding Galois-equivariant embedding is
\[
	\operatorname{Aut}(Q)\xrightarrow[\text{by }B]{\text{conj.}}\operatorname{Aut}(Q_0)\xrightarrow{\sim}\operatorname{Aut}(\mathbb{P}^1)\xrightarrow{\psi}\operatorname{GL}_3(\overline{k})\xrightarrow[\text{by }B^{-1}]{\text{conj.}}\operatorname{GL}_3(\overline{k}),
\]
where we consider $\Aut(Q)$ and $\Aut(Q_0)$ as subgroups of $\Aut(\mathbb{P}^2_{\overline{k}}) \cong \operatorname{PGL}_3(\overline{k})$.
	\end{remark}

\subsubsection{Hilbert's Theorem 90}\label{sect:Hilbert90}
Hilbert's Theorem $90$ is the statement that, for every integer $n \geq 1$ and for every Galois extension of fields $L/k$, the first cohomology set $\operatorname{H}^1(\Gal(L/k),\GL_n(L))$ is trivial. Thus, given a cocycle $\overline{\xi}:\,\operatorname{Gal}(L/k)\to\operatorname{GL}_n(L)$, there exists a matrix $M \in \GL_n(L)$ such that $\overline{\xi}(\sigma)=M \cdot\,^\sigma M^{-1}$. Moreover, this matrix $M$ can be explicitly computed \cite[p.159, Prop. 3]{SerCL}, for instance by choosing a sufficiently generic $M_0 \in \GL_n(L)$ and setting
\begin{equation}\label{eq:Thm90}
M:=\sum_{\sigma \in \Gal(L/k)} \overline{\xi}(\sigma) \cdot {}^{\sigma}M_0.
\end{equation}
Another approach (used in Section $3$ in \cite{Lor15}) is to choose the columns of the matrix $M$ to be a $k$-basis of the $k$-vector subspace of $L^n$ fixed by the action of $\Gal(L/k)$ given by
$$
\sigma:\,v\mapsto \overline{\xi}(\sigma) \cdot {}^{\sigma}v.
$$

\subsubsection{Twisting the underlying conic}
We now describe our algorithm for twisting hyperelliptic curves of odd genus. Recall that our input consists of a cocycle $\xi : \Gal(L/k) \to \Aut_L(C)$; composing $\xi$ with the Galois-equivariant embedding of $\Aut(C)$ in $\GL_3(L)$ provided by Theorem \ref{thm-cle} (and made explicit in Remark \ref{explicit-lift}) we obtain a cocycle $\overline{\xi} : \Gal(L/k) \to \GL_3(L)$.
Applying the explicit form of Hilbert's theorem 90 just described we obtain a matrix $M \in \operatorname{GL}_3(L)$ such that $\overline{\xi}(\sigma) = M \cdot {}^\sigma M^{-1}$ for all $\sigma \in \Gal(L/K)$. We set $Q'(x,y,z)=Q(M(x,y,z))$; using the fact that the matrix $\overline{\xi}(\sigma)$ preserves the equation of $Q$ for every $\sigma \in \Gal(L/k)$, one checks easily that $Q'(x,y,z)$ has $k$-rational coefficients.
Our twist will be described in terms of $Q'$ and of the polynomial $f(M(x,y,z))$; the latter, however, is in general not defined over $k$.
In order to resolve this inconvenience, we show the following:

\begin{proposition}\label{rationalmodel}
	There exist $v\in L^\times$ and $h\in L[x,y,z]$
	such that $vf(M(x,y,z))+h\cdot Q'$ has coefficients in $k$.
\end{proposition}
\begin{proof} As $M \cdot {}^\sigma M^{-1}$ induces an automorphism of $C$, we see that for every $\sigma$ there exist $a_\sigma \in L^\times$ and $p_\sigma \in L[x,y,z]$ such that ${}^\sigma(f(M(x,y,z))=a_\sigma f(M(x,y,z))+(p_\sigma \cdot Q)(M(x,y,z))$. It is easy to check that the map $\sigma \mapsto a_\sigma$ is a (continuous) cocycle of $\operatorname{Gal}(\overline{k}/k)$ with values in $\overline{k}^\times$. Let $v \in \overline{k}^\times$ be such that $a_\sigma=v/{}^\sigma v$ (such a $v$ can again be computed explicitly thanks to the effective form of Hilbert's theorem 90: notice that since $a_\sigma$ takes values in $L^\times$, $v$ can also be taken in $L^\times$).
Now let $m \in L$ be an element such that $\operatorname{tr}_{L/k}(m)=1$ (if $\operatorname{char}(k)$ does not divide $[L:k]$ one can take $m=[L:k]^{-1}$) and set $h(x,y,z)=\sum_{\tau}\,^{\tau}m \, ^{\tau}v \,  p_\tau(M(x,y,z))$.
Multiply the identity ${}^\sigma(f(M(x,y,z))=a_\sigma f(M(x,y,z))+(p_\sigma \cdot Q)(M(x,y,z))$ by ${}^\sigma(mv)$ and sum over $\sigma \in \Gal(L/k)$:
\[
\sum_{\sigma} {}^\sigma(mv) {}^\sigma(f(M(x,y,z)) = \sum_{\sigma} \left( {}^\sigma(mv) a_\sigma f(M(x,y,z)) + {}^\sigma(mv) (p_\sigma \cdot Q)(M(x,y,z)) \right).
\]
The left hand side has coefficients in $k$; we now show that the right hand side is $vf(M(x,y,z))+h \cdot Q'$. Indeed, recalling that $a_\sigma = v/{}^\sigma v$ we obtain 
\[
\begin{aligned}
\sum_{\sigma} & \left( {}^\sigma(mv) a_\sigma f(M(x,y,z))  + {}^\sigma(mv)  (p_\sigma \cdot Q)(M(x,y,z)) \right) = \\
& = \sum_{\sigma} {}^\sigma m \, vf(M(x,y,z)) + Q'(x,y,z)\sum_{\sigma} {}^\sigma m {}^\sigma v \; p_\sigma(M(x,y,z)) \\
& = \left(\sum_{\sigma} {}^\sigma m \right) vf(M(x,y,z)) + h(x,y,z) \, Q'(x,y,z)  \\
& = vf(M(x,y,z)) + h(x,y,z) \, Q'(x,y,z).
\end{aligned}
\]
\end{proof}

\begin{corollary}\label{corollary:AlmostCorrectCocycle}
Let $g(x,y,z) = vf(M(x,y,z))+h\cdot Q'$ be as above. The curve
\[
C_0 : \begin{cases}
t^2 = g(x,y,z)\\
Q'(x,y,z)=0
\end{cases}
\]
is a $k$-twist of $C$. An isomorphism $\psi_0:C_0 \to C$ is given by $(x,y,z,t) \mapsto (M(x,y,z),  t/\sqrt{v})$.
For all $\sigma \in \operatorname{Gal}(\overline{k}/k)$, the automorphisms $\psi_0 \circ {}^\sigma \psi_0^{-1}$ and $\xi(\sigma)$ of $C$ differ at most by the hyperelliptic involution.
\end{corollary}
\begin{proof}
The only nontrivial statement is that $\psi_0 \circ {}^\sigma \psi_0^{-1}$ and $\xi(\sigma)$ differ at most by the hyperelliptic involution, and this follows from the fact that $\xi(\sigma)$ and $M \cdot \, {}^\sigma M^{-1}$ induce the same automorphism on the conic $C/\langle \iota \rangle$.
\end{proof}

\subsubsection{Pinning down the quadratic twist}\label{sect:FinishProof}

In view of Corollary \ref{corollary:AlmostCorrectCocycle}, we can define a new cocycle $\xi_0(\sigma):=\psi_0 \circ {}^\sigma \psi_{0}^{-1}$, which differs from $\xi(\sigma)$ at most by the hyperelliptic involution. Define $i:\operatorname{Gal}(\overline{k}/k) \to \langle \iota \rangle$ by the prescription
$
\xi_0(\sigma)=\xi(\sigma) i(\sigma).
$

\begin{lemma}\label{quadtwist}
The map $i$ is a group homomorphism. By Galois correspondence, its kernel defines an at most quadratic extension $k(\sqrt{e})$ of $k$.
\end{lemma} 

\begin{proof}
Since $\xi$ and $\xi_0$ are both cocycles, so is $\sigma \mapsto i(\sigma)=\xi(\sigma)^{-1}\xi_0(\sigma) \in \langle \iota \rangle$. Since the Galois action on $\langle \iota \rangle$ is trivial, a cocycle is the same thing as a group homomorphism.
The second statement follows from the fact that the image of $i$ has order at most 2.
\end{proof}

\begin{theorem}\label{thm:RestatementMain}
Let $e$ be as in the previous lemma. The curve
\[
C' \, : \begin{cases}
et^2=g(x,y,z) \\
Q(M(x,y,z))=0
\end{cases}
\]
is the $k$-twist of $C$ corresponding to the cocycle $\xi$. The map $\phi : C' \to C$ that sends $(x,y,z,t)$ to $(M(x,y,z), \sqrt{e/v}t)$ realizes the cocycle $\xi$.
\end{theorem}
\begin{proof}
We have $\phi=\psi_0 \circ \psi_1$, where $\psi_1$ is the isomorphism $C' \to C_0$ given by $\psi_1(x,y,z,t)=(x,y,z,\sqrt{e}t)$. It is clear that $C'$ and $C_0$ are isomorphic through $\psi_1$, so (in view of Corollary \ref{corollary:AlmostCorrectCocycle}) $C'$ is a $k$-twist of $C$. It remains to show that $\phi$ realizes the cocycle $\xi$. Notice that by construction one has $\psi_1 \circ  {}^\sigma \psi_1^{-1}=i(\sigma)$, where $i$ is the homomorphism of the previous lemma. Further observing that the image of $i$ is contained in $\langle \iota \rangle$, hence it is central in $\Aut(C)$, we have 
\[
\begin{aligned}
\phi \circ {}^\sigma \phi^{-1} & = \psi_0 \psi_1 \circ {}^\sigma (\psi_0 \psi_1)^{-1} \\
& = \psi_0 \circ \psi_1 \circ  {}^\sigma \psi_1^{-1}\circ  {}^\sigma \psi_0^{-1} \\
& = \psi_0 \circ  i(\sigma) \circ  {}^\sigma \psi_0^{-1} \\
& = \psi_0 \circ  {}^\sigma \psi_0^{-1} \circ i(\sigma) \\
& = \xi_0(\sigma) \circ  i(\sigma) \\
& = \xi_0(\sigma) \circ  i(\sigma)^{-1} =\xi(\sigma),
\end{aligned}
\]
where $i(\sigma)=i(\sigma)^{-1}$ since $\iota$ is of order 2.
\end{proof}

Putting everything together, we have proven the correctness of the following algorithm:

\vspace{5mm}

\noindent\fbox{\begin{minipage}{39.5em}
\vspace{2mm}
\begin{algorithm}\label{algo:Main} (For computing twists of hyperelliptic curves of odd genus)

\vspace{5mm}

\textbf{Input:} A hyperelliptic curve $C/k$  and a cocycle $\xi:\operatorname{Gal}(\overline{k}/k) \to \Aut(C)$ factoring through a finite extension $L/k$.

\textbf{Output:} A curve $C'/k$ and a $\overline{k}$-isomorphism $\phi:\,C'\rightarrow C$ such that  $\xi(\sigma)=\phi\circ\,^{\sigma}\phi^{-1}$ for all $\sigma \in \Gal(\overline{k}/k)$.

\begin{enumerate}
\item Find a model of $C$ of the form \eqref{eq:NonHyperellipticModel} and let $Q:\,z^2=ax^2+by^2$.
\item Project $\xi$ into $Z^1(\operatorname{Gal}(\overline{k}/k),\Aut(Q))$.
\item Compose with the embedding of $\Aut(Q)$ in $\GL_3(\overline{k})$ given by Remark \ref{explicit-lift} to obtain a cocycle $\overline{\xi}:\,\operatorname{Gal}(\overline{k}/k)\rightarrow\operatorname{GL}_3(\overline{k})$.
\item Use Hilbert's theorem 90 to compute a matrix $M \in \operatorname{GL}_3(\overline{k})$ such that $\overline{\xi}(\sigma)=M\, \cdot \, ^\sigma \hspace{-1mm} M^{-1}$.
\item Use Proposition \ref{rationalmodel} to compute a $k$-rational model of 
$$
\begin{cases} t^2=f(M(x,y,z))\\ Q'(x,y,z):=Q(M(x,y,z))=0\end{cases}
$$
of the form
$$
C_0:\,\begin{cases} t^2=g(x,y,z)=vf(M(x,y,z))+h(x,y,z)Q'(x,y,z)\\ Q'(x,y,z)=0\end{cases}.
$$
\item Compute $e$ as in Lemma \ref{quadtwist}. The twist of $C$ by $\xi$ is
$$
C':\,\begin{cases} et^2=g(x,y,z)\\ 0=Q'(x,y,z)\end{cases}.
$$
{The map $\phi:\,(x,y,z,t) \mapsto (M(x,y,z), \sqrt{e/v}t)$ is an isomorphism between $C'$ and the model of $C$ chosen in part (1), and it realizes $\xi$.}
\end{enumerate}
\end{algorithm}
\end{minipage}}

\vspace{5mm}

\section{Examples}\label{sect:Examples}
In this section we explicitly compute some twists of hyperelliptic curves; {this will show how the different parts of the algorithm come together, and why all the steps are necessary.}

\subsection{Example 1}\label{sect:Example2}
Consider the curve $C:y^2=x^9-x\; /\mathbb{Q}$ of even genus $g=4$, and the cocycle
$
\xi : \operatorname{Gal}(\overline{\mathbb{Q}}/\mathbb{Q}) \to \Aut(C)
$
which factors through $\Gal(\mathbb{Q}(i)/\mathbb{Q})$ and sends the generator $\tau$ of this group to the automorphism 
\[
\begin{array}{cccc} \alpha: & C & \to & C \\ & (x,y) & \mapsto & \left(\frac{1}{x}, \frac{iy}{x^5} \right) \end{array}.
\]
We apply Algorithm \ref{algo:MainEvenGenus} to this situation.

\textbf{Steps 1.}
%
The M\"obius transformation $x \mapsto \frac{1}{x}$ is represented by the projective class of matrices of the form $\begin{pmatrix}
0 & \lambda \\
\lambda & 0
\end{pmatrix}$. To find a matrix in $\operatorname{GL}_2(\mathbb{Q}(i))$ which represents $\alpha$ we use the recipe given in the proof of Lemma \ref{superembedding0}: the lift is $\lambda\begin{pmatrix}
0 & 1 \\
1 & 0
\end{pmatrix}$ with $\lambda^{2r-5}=i \cdot (-1)^r$ (here $r=g/2=2$), hence $\lambda=-i$. Thus in particular the lift of our cocycle to $\operatorname{GL}_2(\mathbb{Q}(i))$ sends $\tau$ to $\begin{pmatrix}
0 & -i \\
-i & 0
\end{pmatrix}$.


\textbf{Step 2.} We use the explicit version of Hilbert's Theorem 90 (see subsection \ref{sect:Hilbert90}) with $M_0=\operatorname{Id}$ to get $M=\begin{pmatrix}1 & -i \\ -i & 1 \end{pmatrix}$.

\textbf{Step 3.} We have
\[
C' : \; y^2 = \det(M)^{-4} (-i x+1)^{10}f\left( \frac{x-i}{-ix+1} \right) = -x^9+6 x^7-6 x^3+x
\]
and $M$ induces the isomorphism 
\[
\begin{array}{cccc}
\phi :\, & C' & \to & C \\
& (x,y) & \mapsto & \left(\frac{x-i}{-ix+1},\frac{2^2y}{(-ix+1)^5} \right)=\left(i\frac{x-i}{x+i},\frac{4iy}{(x+i)^5}\right).
\end{array}
\]


\subsection{Example 2}

Consider the curve $C \, : y^2=x(x^6+3x^4+1)$ over $\mathbb{Q}$. The geometric automorphism group of $C$ is cyclic of order 4, generated by $\alpha : (x,y) \mapsto (-x,iy)$. Since the automorphism $i \mapsto -i$ maps $\alpha$ to $\alpha^{-1}$, we can define a cocycle $\Gal(\mathbb{Q}(i)/\mathbb{Q}) \to \Aut(C)$ by sending the generator $\tau$ of $\Gal(\mathbb{Q}(i)/\mathbb{Q})$ to $\alpha$. Let us compute the twist of $C$ corresponding to this cocycle with Algorithm \ref{algo:Main}. 

\textbf{Step $1.$} A model as in \eqref{bigmodel} is given by
$
\begin{cases}
t^2=f(x,y,z):=x^3z+3x^2yz+y^3z \\
z^2=xy
\end{cases}.
$

\textbf{Steps $2$ and $3.$} The $2 \times 2$ matrix representing $\alpha$ in the hyperelliptic model is $\begin{pmatrix}
\zeta_8 & 0 \\ 0 & -\zeta_8
\end{pmatrix}$. Applying the embedding of Remark \ref{CoolEmbedding} we obtain that the corresponding element in $\GL_3(\overline{\mathbb{Q}})$ is simply $\overline{\xi}(\tau):=\operatorname{diag}(-1,-1,1)$. 
Notice that this matrix has order 2 while $\alpha$ has order 4, but this is to be expected: our embedding only considers the automorphism induced on $C/\langle \iota \rangle$, and $\alpha$ squares to the hyperelliptic involution, which is trivial on $C/\langle \iota \rangle$.

\textbf{Step 4.} We find $M=\operatorname{diag}(i,i,1)$. 

\textbf{Step 5.}
We compute $(z^2-xy)(M(x,y,z))=(z^2-xy)(ix,iy,z)=z^2+xy$ and
\[
\begin{aligned}
f(M(x,y,z)) & =f(ix,iy,z) & =-i f(x,y,z);
\end{aligned}
\]
it follows that ${}^{\tau}(f(M(x,y,z))=-f(M(x,y,z))$ so that $(a_\sigma)$ is the cocycle mapping $\tau \mapsto -1$ and $(p_\sigma)=0$. We can then take $v=i$, $h=0$, and $g(x,y,z)=if(M(x,y,z))=f(x,y,z)$. 
The intermediate curve $C_0$ is therefore $C_0 \, : \begin{cases} t^2=f(x,y,z) \\ z^2=-xy \end{cases}.$ 

\textbf{Step 6.} Let $\sigma_i$, for $i=1,3,5,7$, be the automorphism of $\mathbb{Q}(\zeta_8)/\mathbb{Q}$ which sends $\zeta_8$ to $\zeta_8^i$. We have that $\sigma_1, \sigma_5$ map to the identity in $\Gal(\mathbb{Q}(i)/\mathbb{Q})$, while $\sigma_3, \sigma_7$ map to $\tau$. One checks that the homomorphism $i$ of Lemma \ref{quadtwist} is trivial on $\sigma_1, \sigma_3$ and nontrivial on $\sigma_5, \sigma_7$. This means that we can take $e=-2$, because $\mathbb{Q}(\sqrt{-2})$ is the fixed field of $\sigma_3$. 
Finally, the curve we are looking for is
\[
C': \, \begin{cases} -2t^2=f(x,y,z) \\ z^2=-xy \end{cases}
\]
and an isomorphism realizing the cocycle $\xi$ is $\phi \, : (x,y,z,t) \mapsto (ix,iy,z,(1+i)t)$. 

Since the conic $z^2=-xy$ obviously has rational points, $C'$ also admits a $\mathbb{Q}$-hyperelliptic model: one can check that such a model is given by $y^2=2x^7 + 6x^3 - 2x$, and on this model an isomorphism $C' \to C$ realising the cocycle is $(x,y) \mapsto \left(\frac{i}{x}, -\frac{1}{2}(i + 1)\frac{y}{x^4} \right)$.

\subsection{Example 3}
Finally, we give an example where $C$ admits a hyperelliptic model over its field of definition but one of its twists does not.
Consider the curve $C:v^2=u^8+14u^4+1 \; /\mathbb{Q}$ and the cocycle
$
\xi : \operatorname{Gal}(\overline{\mathbb{Q}}/\mathbb{Q}) \to \Aut(C)
$
which factors through $\Gal(\mathbb{Q}(i)/\mathbb{Q})$ and sends the generator $\tau$ of this group to the automorphism 
\[
\begin{array}{cccc} \alpha: & C & \to & C \\ & (u,v) & \mapsto & \left( -\frac{1}{u}, -\frac{v}{u^4} \right) \end{array}.
\]

\smallskip

\textbf{Step 1.} $C$ admits the model
$
\; C: \; \begin{cases}
t^2=x^4+14x^2y^2+y^4 \\
xy=z^2
\end{cases}
$

\textbf{Steps 2 and 3.}
In these coordinates, $\alpha$ is given by $[x:y:z:t] \mapsto [-y:-x:z:-t]$, and we have a cocycle with values in $\Aut(C/\langle \iota \rangle)$ given by $\tau \mapsto \left[ \begin{pmatrix} 0 & -1 & 0 \\ -1 & 0 & 0 \\ 0 & 0 & 1 \end{pmatrix} \right]$. Via Remark \ref{explicit-lift} we obtain a cocycle $\overline{\xi}$ with values in $\GL_3(\mathbb{Q}(i))$ which sends $\tau$ to the matrix $ \begin{pmatrix} 0 & 1 & 0 \\ 1 & 0 & 0 \\ 0 & 0 & -1 \end{pmatrix}$.

\textbf{Steps 4 and 5.} A possible choice of $M$ is $\begin{pmatrix}1+i & 1-i & 0 \\ 1-i & 1+i & 0 \\ 0 & 0 & 2i\end{pmatrix}$. We have $v=1$, and the curve of Corollary \ref{corollary:AlmostCorrectCocycle} is
$
C_0:\,\begin{cases} X^2+Y^2+2Z^2=0 \\ t^2=48X^4+160X^2Y^2+48Y^4
\end{cases}.
$

\textbf{Step 6.} The isomorphism $\varphi : [x:y:z:t] \mapsto [M(x,y,z),t]$ acts trivially on $t$, so we have $\varphi \circ {}^\tau \varphi(t)^{-1}=t$, while $\xi_\tau(t)=-t$. It follows that $e=-1$, and the twist of $C$ corresponding to $\xi$ is 
\[
C'\, : \begin{cases} X^2+Y^2+2Z^2=0 \\ -T^2=48X^4+160X^2Y^2+48Y^4
\end{cases}
\]
Notice that this curve does not have 
a hyperelliptic model over $\mathbb{Q}$, because the conic $X^2+Y^2+2Z^2=0$ has no $\mathbb{Q}$-rational points.
It is also easy to check that the change of variables $A=X-Y, B=X+Y, C=2Z, D=T/4$ leads to the more symmetrical model
$
\begin{cases}A^2+B^2+C^2=0 \\ -2D^2 = A^4+B^4+C^4\end{cases}
$,
which agrees with the result found by a different method in \cite{AWS2016}.

\bibliographystyle{abbrv}
\bibliography{biblio}

\end{document}